\documentclass[11pt]{article}

\usepackage[margin=1in]{geometry}
\usepackage{hyperref}

\usepackage{graphicx}    
\usepackage{amsmath,amsthm,amsfonts,amssymb,amscd}
\usepackage{array,multirow,longtable,hhline}
\usepackage{mathrsfs}
\usepackage[dvipsnames,svgnames,x11names]{xcolor}
\usepackage[numbers]{natbib}

\newtheorem{theorem}{Theorem}[section]
\newtheorem{lemma}[theorem]{Lemma}

\newtheorem{corollary}[theorem]{Corollary}

\theoremstyle{definition}

\newtheorem{example}[theorem]{Example}

\theoremstyle{remark}
\newtheorem{remark}[theorem]{Remark}
\newtheorem{prb}[theorem]{Problem}

\numberwithin{equation}{section}

\usepackage{authblk}

\title{Local Antimagic Chromatic Number of Middle Graph and Direct Product of Graphs}
\author{
	Huimin Xu\thanks{School of Mathematics and Statistics, Qingdao University, Qingdao 266071, China. \texttt{huimin\_xu0901@163.com}}
	\and
	Gee-Choon Lau\thanks{Corresponding author. 77D, Jalan Suboh, 85000 Johor, Malaysia. \texttt{geechoonlau@yahoo.com}}
	\and
	Ruixue Zhang\thanks{School of Mathematics and Statistics, Qingdao University, Qingdao 266071, China. \texttt{ruixuezhang7@163.com}}
}
	
\begin{document}

	\maketitle
	\begin{abstract}
		An edge labeling of a connected graph $G=(V, E)$ is called a local antimagic labeling if it is a bijection $f: E\rightarrow \{1, 2, \cdots, \left | E \right | \}$ such that for any pair of adjacent vertices $u$ and $v$, $f^{+}(u)\neq f^{+}(v)$ where $f^{+}(x)=\sum f(e)$, with $e$ ranging over all the edges incident to $x$. The local antimagic chromatic number of $G$, denoted $\chi_{la}(G)$, is the minimum number of distinct induced vertex labels across all possible local antimagic labelings of $G$. This paper investigates two families of graphs $G(m^{[n]})$ (a graph class constructed by augmenting a bipartite structure $B(m^{[n]})$ with cyclic edges) and $K_{1,s}\times C_n$ (the direct product of a star $K_{1,s}$ and a cycle $C_n$). Under the conditions  odd integer $m\ge 1$ and even integer $n\ge 4$, we establish the exact values of $\chi_{la}(G(m^{[n]}))$ and $\chi_{la}(K_{1,s}\times C_n)$ for various $s$. Furthermore, we characterize several graph families achieving $\chi_{la}(G)=2$ or $3$.
	\end{abstract}
	



\section{Introduction}
Let $G=(V, E)$ be a connected simple graph of order $p$ and size $q$. Arumugam et al. \cite{arumugam2017local} introduced the concept of local antimagic labeling and local antimagic chromatic number.
A bijection $f: E\rightarrow \{1, 2, \cdots, q\}$ is called a \textit{local antimagic labeling}
if $f^{+}(u)\neq f^{+}(v)$ for any edge $uv\in E$, 
where $f^{+}(x)=\sum_{e\in E(x)}f(e)$ and $E(x)$ is the set of edges incident to $x$.
The \textit{color number} of a local antimagic labeling $f$ is the number of distinct induced colors under $f$, denoted by $c(f)$.
Such a labeling $f$ is referred to as a \textit{local antimagic $c(f)$-coloring}.
The \textit{local antimagic chromatic number} $\chi_{la}(G)$ is the 
minimum $c(f)$ over all possible local antimatic labelings of $G$. Haslegrave resolved a conjecture posed in \cite{arumugam2017local}, proving that every connected graph other than $K_2$ admits a local antimagic labeling \cite{haslegrave2018proof}. Thus, every graph without a $K_2$ component admits a local antimagic labeling.

Let $G$ and $H$ be two connected graphs.
The \textit{disjoint union} of $G$ and $H$, denoted $G+H$, is the graph with vertex set $V(G)\cup V(H)$ and edge set $E(G)\cup E(H)$.
The \textit{join} of $G$ and $H$, denoted $G\vee H$, is the graph with vertex set $V(G)\cup V(H)$ and edge set $E(G)\cup E(H)\cup \left\{(x,y)\mid x\in V(G), y\in V(H)\right\}$. 
The bounds of the local antimagic chromatic number for disjoint union of multiple copies of a graph are estimated in \cite{math9111230}.
Investigations into the local antimagic chromatic number for the join of certain graphs can be found in \cite{yang2021local, premalatha2023local, lau2021local}. The local antimagic crhomatic number for various families of graphs, including regular graphs, spider graphs, and graphs with cut-vertices, is determined in
\cite{lau2023join, lau2022local, lau2024local} respectively. 
In \cite{lau2020affirmative}, Lau et al. established the local antimagic chromatic number of complete bipartite graphs.

For positive integers $m$ and $n$ ($\ge 2$), let $B(m^{[n]})$ be a bipartite graph with bipartition
\begin{equation*}
	\left(\bigcup_{i=1}^n X_i, \bigcup_{i=1}^n V_i\right)
\end{equation*}
and  where  $\bigcup_{i=1}^nX_i = \left\{u_1,u_2,\cdots,u_n\right\}$, $X_i=\left\{u_i,u_{i+1}\right\}$ (with $u_{n+1}=u_1$) and $V_i = \left\{v_{i,1}, v_{i,2}, \cdots, v_{i,m}\right\}$ for $i=1,2, \cdots, n$. The edge set of $B(m^{[n]})$ is $\bigcup_{i=1}^nE_i$, where $E_i=E(K_{2,m}^i)$ and $K_{2,m}^i$ is a complete bipartite graph with bipartition $(X_i, V_i)$. For example, $B(1^{[4]})$ is in Figure~\ref{fig: graph1}. Notably, $B(1^{[n]}) \cong C_{2n}$ and $B(m^{[2]}) \cong K_{2,2m}$. In~\cite{lau2021local}, Lau et al. determined the exact values of the local antimagic chromatic number of $B(m^{[n]})$ for various values of $m$ and $n$. 

For positive integers $m$ and $n$ $(\ge 3)$, let $G(m^{[n]})$ be the graph obtained from $B(m^{[n]})$ by adding $n$ edges joining vertices $u_i$ and $u_j$ such that (i) $i \not\equiv j \pmod{2}$ with the exception that if $n$ is odd, there is exactly one odd $i\geq 3$ where $u_1$ is adjacent to $u_i$, and (ii) the $n$ edges induce an $n$-cycle. For example, $G(1^{[4]})$ is shown in Figure~\ref{fig: graph2}.
Without loss of generality, we will always assume that the added edges induce an $n$-cycle consisting of consecutive edges $e_1,e_2,\ldots,e_n$, where $u_1$ is incident to edges $e_1$ and $e_n$, unless otherwise stated. 

The \textit{middle graph} $M(G)$ of a graph $G$ \cite{yan2017enumeration} 
is defined with
the vertex set $V(G)\cup E(G)$ and 
two vertices $x, y\in V(M(G))$ are adjacent in $M(G)$ if either (i) $x, y\in E(G)$, and $x, y$ are adjacent in $G$; or (ii)$x\in V(G), y\in E(G)$, and $x, y$ are incident in $G$. Note that $G(1^{[n]})$ is also the middle graph of $C_n$ if the induced $n$-cycle is $u_1u_2\cdots u_n$.

\begin{figure}[ht]
\centering
\begin{minipage}{0.35\textwidth}
    \centering
    \includegraphics[width=\textwidth]{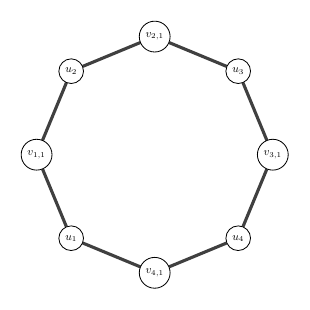}
    \caption{$B(1^{[4]})$}
    \label{fig: graph1}
\end{minipage}
\begin{minipage}{0.35\textwidth}
    \centering
    \includegraphics[width=\textwidth]{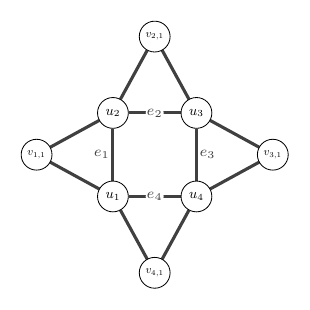}
    \caption{$G(1^{[4]})$}
    \label{fig: graph2}
\end{minipage}
\end{figure}

Let $G=(V(G),E(G))$ and $H=(V(H),E(H))$ be two graphs. The direct product, denoted by $G\times H$, is a graph with vertex set $V(G)\times V(H)=\left\{(u,v)\mid u\in V(G),v\in V(H)\right\}$ and edge set $E(G\times H)=\{\left(\left(u,v\right), \left(u',v'\right)\right)\mid$ $ (u,u')\in(E(G)) \text{ and } (v,v')\in(E(H))\}$. The direct product $K_{1,s}\times C_n$ with even $n\ge 4$ is also related to $B(s^{[n/2]})$.

In this paper, we show that graphs $G(m^{[n]})$ with odd $m\ge 1$ and even $n\ge 4$ have a local antimagic chromatic number of $3$. Additionally, the direct product $K_{1,s}\times C_n$ has a local antimagic chromatic number $2$ for even $n\ge 4$, even $s\ge 2$ or $s = n/2$. Moreover, certain graphs with local antimagic chromatic numbers $2$ or $3$ are attained by splitting or merging of vertices. To the best of our knowledge, this is the first paper to successfully determine the local antimagic chromatic number of the middle graph of a graph family and the direct product of two graph families.

\section{Preliminaries}
In this section, we shall introduce some basic knowledge and lemmas which will be used to prove our main results later. By definition, we have for any non-empty graph $G$ without any $K_2$ component, $\chi_{la}(G)\ge \chi(G) \ge 2$ where $\chi(G)$ is the chromatic number of $G$.

\begin{lemma}\label{lem-Gmn-tri} For any positive integer $m$ and even $n\ge 4$, $\chi(G(m^{[n]})) = 3$. 
\end{lemma}

\begin{proof} 
Suppose that $\chi(G(m^{[n]})) = 2$, thus there exists $j\in \{1, 2, \ldots, m\}$ such that the cycle $u_1v_{1,j}u_2v_{2,j}u_{3}v_{3,j}$ $\cdots u_{n}v_{n,j}u_1$
must have the vertices $u_i$ (for $i=1,2,\cdots,n$) all belonging to the same color set, while $v_{i,j}$ (for $i=1,2,\cdots,n$) must be in another color set.
This leads to a contradiction,
as there are edges among the vertices $u_i$ (for $i=1,2,\cdots,n$), implying that they must belong to different color sets. Therefore, $\chi(G(m^{[n]}))\ge 3$.  	 
By definition of $G(m^{[n]})$, $G(m^{[n]})$ have $3$-independent partitions given by $V = \{v_{i,1}, \ldots, v_{i,m} \mid 1\le i\le n\}$, $U_1 = \{u_{2i-1}\mid 1\le i\le n/2\}$ and $U_2 = \{u_{2i}\mid 1\le i\le n/2\}$, which implies $\chi(G(m^{[n]})) \le 3$. 
The proof is completed.
\end{proof} 

\begin{lemma}\label{lem-B}\cite{lau2021local}
 Let $s\ge 2$ and $n\ge 3$ be positive integers. For any even $s$, $\chi_{la}(B(s^{[n]}))=2$.
\end{lemma}

\begin{lemma}\label{lem-q+1-f}\cite{lau2021local}
Let $G$ be a graph of size $q$ and $f$ be a local antimagic labeling of $G$. 
For any vertices $x, y\in V(G)$, if
	
	\begin{itemize}
		\item[(i)] $f^+(x)=f^+(y)$ implies that $deg(x) = deg(y)$, and
		\item[(ii)] $f^+(x) \ne f^+(y)$ implies that $(q+1)(deg(x)-deg(y))\ne f^+(x)-f^+(y)$,
	\end{itemize}
where $deg(x)$ is the degree of $x$ in $G$,
then the mapping $g: e\rightarrow q+1-f(e)$
is also a local antimagic labeling of $G$ with $c(f)=c(g)$, where $e\in E(G)$.
\end{lemma}

\begin{lemma}\label{lem-delete-edge}\cite{lau2021local}
Let $H$ be a graph obtained from $G$ by deleting an edge $e$ and $f$ be a local antimagic labeling of $G$ with $f(e)=1$.
For positive integer $t\geq 2$, if $G$ satisfies the following conditions:
	\begin{itemize}
		\item[(i)] $\chi_{la}(G) = t$ and $f$ is a local antimagic labeling of $G$ that induces a $t$-independent partition $\bigcup_{i=1}^{t}V_i$ of $V(G)$;
		\item[(ii)] for each vertex $x\in V_k$ $(1\le k\le t)$, degree of $x$ is constant $d_k$;
		\item[(iii)] $f^+(u)-d_a \ne f^+(v)-d_b$, where $u\in V_a$ and $v\in V_b$ for $1\le a \ne b\le t$,
	\end{itemize}
     then $\chi(H)\le \chi_{la}(H)\le t$.
\end{lemma}

Moreover, the numbers discussed in this paper are integers, unless otherwise specified.

\section{Local antimagic chromatic number of $G(m^{[n]})$ for odd $m\ge 1$ and even $n\ge 4$}


In this section, we shall calculate the local antimagic chromatic number of $G(m^{[n]})$
when $m$ is odd $(m\ge 1)$ and $n$ is even $(n\ge 4)$. We will also obtain some families of tripartite graphs with local antimagic chromatic number $3$. Let $\delta=n/2$.

\begin{theorem}\label{th-Gmn}
	For odd $m\ge 1$ and even $n\ge 4$, $\chi_{la}(G(m^{[n]})) = 3$.
\end{theorem}

\begin{proof}
    By Lemma~\ref{lem-Gmn-tri}, $\chi_{la}(G(m^{[n]}))\ge \chi(G(m^{[n]}))= 3$.\\ Define $f: E(G(m^{[n]}))\to \left\{1,2,\cdots, (2m+1)n\right\}$ such that
    \begin{itemize}
        \item[(1)] $f(e_{2i-1}) = n-(2i-2) \text{~for~} 1\le i\le \delta;$
        \item[(2)] $f(e_{2i}) = 2i-1 \text{~for~} 1\le i\le \delta;$
        \item[(3)] $f(u_iv_{i,2j-1}) =n+i+n(2j-2)\text{~for~} 1\le i\le n, 1\le j\le\dfrac{m+1}{2};$
        \item[(4)] $f(u_iv_{i,2j}) =(2m+1)n-n(2j-1)-(i-1),\\ f(u_{i+1}v_{i,2j}) =n+i+n(2j-1) \text{~for~} 1\le i\le n, 1\le j\le\dfrac{m-1}{2};$
        \item[(5)] $f(u_{i+1}v_{i,2j-1}) =(2m+1)n-n(2j-2)-(i-1)\text{~for~} 1\le i\le n, 1\le j\le\dfrac{m+1}{2};$
        
    \end{itemize}
    where $u_{n+1}=u_1$. This bijection indicates that
    \begin{flalign*}
    	\begin{aligned}
    		f^+(v_{i,2j-1})&=f(u_iv_{i,2j-1})+f(u_{i+1}v_{i,2j-1})\\
    		&=n+i+n(2j-2)+(2m+1)n-n(2j-2)-(i-1)\\
    		&=2nm+2n+1\text{~for~} 1\le i\le n\text{~and~} 1\le j\le\dfrac{m+1}{2},\\
    		f^+(v_{i,2j})&=f(u_iv_{i,2j})+f(u_{i+1}v_{i,2j})\\
    		&=(2m+1)n-n(2j-1)-(i-1)+n+i+n(2j-1)\\
    		&=2mn+2n+1\text{~for~} 1\le i\le n\text{~and~} 1\le j\le\dfrac{m-1}{2},\\
            f^+(u_{2i-1})&=f(e_{2i-1})+f(e_{2i})+\sum_{j=1}^{\frac{m+1}{2}}[f(u_{2i-1}v_{2i-1,2j-1})+f(u_{2i-1}v_{2i-2,2j-1})]
    		\\&+\sum_{j=1}^{\frac{m-1}{2}}[f(u_{2i-1}v_{2i-1,2j})+f(u_{2i-1}v_{2i-2,2j})]\\
    		&=n-1+\sum_{j=1}^{\frac{m+1}{2}}(2mn+2n+2)+\sum_{j=1}^{\frac{m-1}{2}}(2mn+2n)\\
    		&=2nm^2+2mn+m+n\text{~for~}2\le i\le \delta,\\
        \end{aligned}
    \end{flalign*}		
    
    \begin{flalign*}
    	\begin{aligned}
    		f^+(u_1)&=f(e_1)+f(e_n)+\sum_{j=1}^{\frac{m+1}{2}}[f(u_1v_{1,2j-1})+f(u_1v_{n,2j-1})]\\&+\sum_{j=1}^{\frac{m-1}{2}}[f(u_1v_{1,2j})+f(u_1v_{n,2j})]\\
    		&=2n-1+\sum_{j=1}^{\frac{m+1}{2}}(2mn+n+2)+\sum_{j=1}^{\frac{m-1}{2}}(2mn+3n)\\
    		&=2nm^2+2mn+m+n,\\
            f^+(u_{2i})&=f(e_{2i})+f(e_{2i-1})+\sum_{j=1}^{\frac{m+1}{2}}[f(u_{2i}v_{2i,2j-1})+f(u_{2i}v_{2i-1,2j-1})]
    		\\&+\sum_{j=1}^{\frac{m-1}{2}}[f(u_{2i}v_{2i,2j})+f(u_{2i}v_{2i-1,2j})]\\
    		&=n+1+\sum_{j=1}^{\frac{m+1}{2}}(2mn+2n+2)+\sum_{j=1}^{\frac{m-1}{2}}(2mn+2n)\\
    		&=2nm^2+2mn+m+n+2\text{~for~} 1\le i\le \delta.
    	\end{aligned}
    \end{flalign*}

    Therefore $f$ is a local antimagic $3-$coloring of $G(m^{[n]})$ so that $\chi_{la}(G(m^{[n]}))\le 3$. This completes the proof.
\end{proof}

\begin{example}
	$\chi_{la}(G(3^{[4]}))=3$, the local antimagic 3-coloring is shown in Figure~\ref{G(3^[4])}.
	\begin{figure}[ht]
        \centering
        \includegraphics[width=0.58\textwidth]{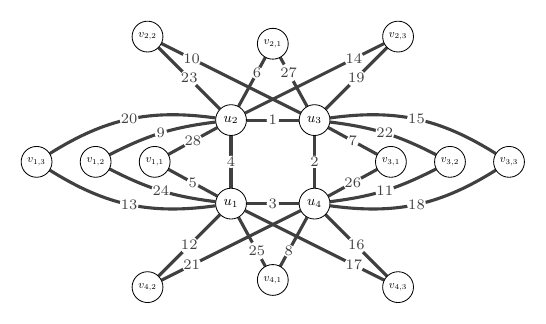}
        \caption{$G(3^{[4]})$}
        \label{G(3^[4])}
    \end{figure}
\end{example}

\begin{corollary}
	For any $e\in E(G(m^{[n]}))$, $\chi_{la}(G(m^{[n]})-e)=3$.
\end{corollary}

\begin{proof}
{\sloppy
	For $G(m^{[n]})$, as shown in Theorem~\ref{th-Gmn}, $f$ is a local antimagic labeling of $G(m^{[n]})$ that induces a $3$-independent partition $V= \left\{v_{i,j}\mid 1\le i\le n, 1\le j\le m\right\}$ , $U_1=\left\{u_{2i-1}\mid 1\le i \le \delta\right\}$, $U_2=\left\{u_{2i}\mid 1\le i \le \delta\right\}$ with label $2nm+2n+1$, $2nm^2+2nm+m+n$ and $2nm^2+2nm+m+n+2$ respectively which satisfies conditions $(ii)$ and $(iii)$ of Lemma~\ref{lem-delete-edge}. Now we prove that it is possible to assign label $1$ to any edge of $G(m^{[n]})$.
	
	By Lemma~\ref{lem-q+1-f}, $g=2nm+n+1-f$ is also a local antimagic labeling of $G(m^{[n]})$ with $c(f) = c(g)$, which with the symmetry of $G(m^{[n]})$ follows that we can assign label $1$ to any edge of $G(m^{[n]})$, by Lemma~\ref{lem-delete-edge}, $\chi(G(m^{[n]})-e)\le \chi_{la}(G(m^{[n]})-e) \le \chi_{la}(G(m^{[n]})) = 3$.

	Suppose $e=u_iu_j$. By the same argument in proving $\chi(G(m^{[n]})) \ge 3$ as in Lemma~\ref{lem-Gmn-tri}, we also can conclude that $\chi(G(m^{[n]})-e) \ge 3$. Suppose $e=u_iv_{i,j}$ or $v_{i,j}u_{i+1}$, the cycle considered in Lemma~\ref{lem-Gmn-tri} is changed to a path. Similarly, we also conclude that $\chi(G(m^{[n]})-e)\ge 3$. The proof is completed.} 

\end{proof}

\section{Graphs with local antimagic chromatic number 3}
In this section, we shall construct graphs whose local antimagic chromatic number are 3 based on $G(m^{[n]})$ by spliting or merging certain vertices.

\sloppy Construct a subset of vertices, let
$SM = \{u_1, u_i \in V(G(m^{[n]}))\mid u_i \mbox{ is incident to } e_{i-1} \mbox{ and } e_i, i \mbox{ is odd in } [3,n]\}$ . If $|SM|\ge 2$, then we can choose any $l\ge 2$ vertices $u_{a_1}, u_{a_2}, \ldots, u_{a_l}$ from $SM$, denoted by $x_1, x_2, \cdots, x_l$ respectively, such that $4\le |a_j-a_{j+1}|\le n-4$ for $1\le j\le l$ with $u_{a_{l+1}} = u_{a_1}$. For $1\le j\le l$, we split each $x_j$ into $x_j'$ and $x_j''$ such that (i) if $x_j \ne u_1$, then $x_j'$ is adjacent to $v_{a_j}$, incident to $e_{a_j}$, and $x_j''$ is adjacent to $v_{a_j-1}$, incident to $e_{a_j-1}$, and (ii) if $x_j = u_1$, then $x_j'$ is adjacent to $v_1$, incident to $e_1$, and $x_j''$ is adjacent to $v_n$, incident to $e_n$.

Merge $x'_j$ and $x''_{j+1}$ into $x^0_{a_j}$ where $j\in [1,l]$, $x''_{l+1} = x''_1$.

We now have a new graph $G_{m,n}$ and the following theorem.

\begin{theorem}\label{th-SM-m} Let $G_{m,n}$ be defined as above, then $\chi_{la}(G_{m,n}) = 3$. 
\end{theorem}

\begin{proof}
	In this way of splitting and merging, $G_{m,n}$ has no multiple edges nor loops. Moreover, every vertex $x_{a_j}^0$ with odd $a_j$ in $[1,n]$, $1\le j\le l$, is adjacent to two 4-degree vertices with even subscripts and two 2-degree vertices.
	
	$V(G_{m,n})$ admits a 3-independent partition 
	\begin{equation*}
		U_1=\left\{u_i \mid i \text{ is odd}, 1\le i\le n, i\ne a_j, 1\le j\le l\right\}\bigcup \left\{x_{a_j}^0\mid 1\le j\le l\right\}, 
	\end{equation*}
	\begin{equation*}
		U_2=\left\{u_i \mid i \text{ is even}, 1\le i\le n\right\},
	\end{equation*}
	\begin{equation*}
		V=\left\{v_{i,j} \mid 1\le i\le n, 1\le j\le m\right\}
	\end{equation*}
which implies that $\chi(G_{m,n})\le 3$. Clearly $\chi(G_{m,n})\ge 2$. Suppose $\chi(G_{m,n}) = 2$, which implies that $V(G_{m,n})$ can be divided into two independent partition $V_1$ and $V_2$. Without loss of generality, every 4-degree vertex with odd subscript must be in $V_1$ and those with even subscript must be
	in $V_2$. Therefore, every 2-degree vertex must be in $V_1$ or in $V_2$, contradicting the fact that every 2-degree vertex must be adjacent to a vertex in $V_1$ and a vertex in $V_2$. Therefore, $\chi(G_{m,n}) = 3$
	which implies $\chi_{la}(G_{m,n}) \ge 3$.
	
	The labeling $f$ for $G(m^{[n]})$ in Theorem~\ref{th-Gmn} admits an edge labeling $g$ of $G_{m,n}$ with 
	\[g^+(v)=f^+(v)\text{~for~} \forall v\in V(G_{m,n})\setminus \left\{x_{a_j}^0\mid 1\le j\le l\right\},\]
	\[g^+(x_{a_j}^0)=2nm^2+2mn+m+n\text{~for~} 1\le j\le l,\]
	due to 
    \begin{equation*}
    	\begin{aligned}
    		&\sum_{j=1}^{\frac{m+1}{2}}[f(u_pv_{p,2j-1})+f(u_qv_{q-1,2j-1})]+\sum_{j=1}^{\frac{m-1}{2}}[f(u_pv_{p,2j})+f(u_qv_{q-1,2j})]\\&+f(e_p)+f(e_{q-1})\\
    		&= [n+p+(2m+1)n-q+2]+\sum_{j=1}^{\frac{m-1}{2}}[4(m+1)n+2]+n-(p-1)+q-2\\
    		&= (2m+2)n+2+(p-q)+2(m^2-1)n+m-1+n-1-(p-q)\\
    		&= 2nm^2+2nm+m+n, \text{ for any }p, q\in [1,n].
    	\end{aligned}
    \end{equation*}
Since
\begin{equation*}
	\begin{aligned}
		f^+(v_{i,j})&=2nm+2n+1\text{~for~} 1\le i\le n\text{~and~} 1\le j\le m,\\
		f^+(u_{2i-1})&=2nm^2+2mn+m+n\text{~for~}1\le i\le n,\\
		f^+(u_{2i})&=2nm^2+2mn+m+n+2\text{~for~} 1\le i\le n,
	\end{aligned}
\end{equation*}
$g$ is a local antimagic $3$-labeling of $G_{m,n}$ which implies $\chi_{la}(G_{m,n})\le 3$. This completes the proof.

\end{proof}

\begin{corollary}
	Let $R$ be a 2-regular graph with only even order components, then $\chi_{la}(M(R))=3$.
\end{corollary}

\begin{proof} 
	If $m=1$ in Theorem~\ref{th-SM-m}, $G_{1,n}$ is obtained from $G(1^{[n]})$ such that $a_1< a_2<\cdots < a_l$ for split vertices $u_{a_1}, u_{a_2},\cdots, u_{a_l}$ while the induced $n$-cycle is $C_n=u_1u_2\ldots u_nu_1$ in $G(1^{[n]})$. Now $G(1^{[n]})$ is a middle graph of  $C_n$. This implies that $G_{1,n}$ is disconnected with $p\ge 2$ components such that each component is a middle graph of a cycle of even order at least 4. If $R$ is connected, then $M(R)=G(1^{[n]})$. By Theorem~\ref{th-Gmn}, $\chi_{la}(M(R))=3$. If $R$ is  disconnected, then $M(R)=G_{1,n}$. By Theorem~\ref{th-SM-m}, $\chi_{la}(M(R))=3$. 
\end{proof} 

\begin{example}
	The graph in  Figure~\ref{G_3,10} is obtained from $G(3^{[10]})$ by split and merge as described above where $SM=\left\{u_1, u_5, u_9\right\}$ and chosen vertices for split are $u_1$ and $u_5$. We get a disconnected graph in family $G_{3,10}$ that contains two components $G(3^{[4]})$ and $G(3^{[6]})$.

    \begin{figure*}[ht]
        \centering
        \begin{minipage}{0.55\textwidth}
            \centering
            \includegraphics[width=\textwidth]{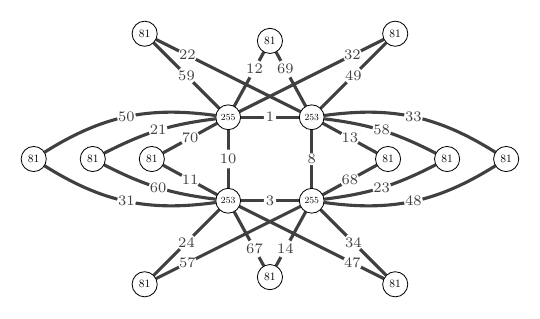}
        \end{minipage}
        \begin{minipage}{0.4\textwidth}
            \centering
            \includegraphics[width=\textwidth]{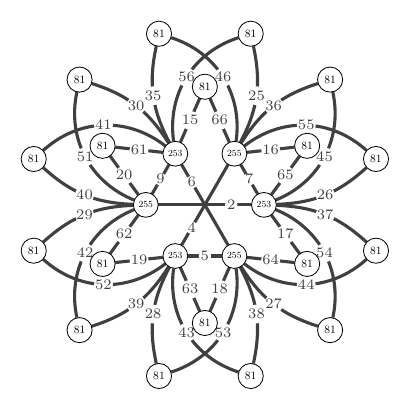}
        \end{minipage}
        \caption{A graph from $G_{3,10}$ with local antimagic 3-coloring}
        \label{G_3,10}
    \end{figure*}

    
\end{example}

After split and merge in Theorem~\ref{th-SM-m}, connected or disconnected graph can be attained. Assume the graph obtained $G_{m,n}$ is disconnected with components $G_1$ and $G_2$ where $m$ is odd ($m\ge 3$) and $n$ is even ($n\ge 4$). Let $U_c^1 = \{u_{2i}\mid u_{2i}\in V(G_1),1\le i\le \delta\}$ and $U_c^2 = \{u_{2j} \mid u_{2j} \in V(G_2), 1\le j\le \delta\}$. Now we continue to consider bijection $g$ in Theorem~\ref{th-SM-m}. Observe that under the bijection $g$, every vertex in $U_c^1$ or in $U_c^2$ is incident to two edges with labels sum $2nm^2+2nm+m+1$ and another two edges with labels sum $n+1$.

Choose and split a vertex $u_{2i}\in U_c^1$ into $y^1_{2i}$ and $y^2_{2i}$ such that $y^1_{2i}$ is incident to the two edges with labels sum $2nm^2+2nm+m+1$ and $y^2_{2i}$ is incident to the two edges with labels sum $n+1$. Similarly, choose and split a vertex $u_{2j}\in U_c^2$ into $y^1_{2j}$ and $y^2_{2j}$. And merge $y^1_{2i}$ and $y^2_{2j}$ into $y_{2j}^0$, $y^2_{2i}$ and $y^1_{2j}$ into $y_{2i}^0$, then new graph is obtained.

We may repeat the same process if there are still vertex not chosen in both of $U^1_c$ and $U^2_c$, and the graph obtained is not isomorphic to all previously obtained graph(s). 

\begin{theorem}
	The graph $H_1$ obtained as described above has local antimagic chromatic number 3.
\end{theorem}

\begin{proof}
	By the proof of Theorem~\ref{th-SM-m}, $V(G_{m,n})$ admits a 3-independent partition, $U_1$, $U_2$ and $V$. Since $H_1$ is obtained from $G_{m,n}$ by splitting certain vertices only in $U_2$ and merging them, $V(H_1)$ still admits a 3-independent partition. Thus $\chi(H_1)\le 3$. Using the same argument as in Theorem~\ref{th-SM-m}, $\chi(H_1) = 3$ which indicates $\chi_{la}(H_1)\ge 3$. 
	
	Keeping all the edge labels, we see that $H_1$ admits a local antimagic labeling $h_1$ with
	\[h_1^+(y_{2i}^0)=h^+(y_{2j}^0)=(2nm^2+2nm+m+1)+(n+1)=2nm^2+2nm+m+n+2\]
	\[h_1^+(v)=g^+(v)\text{~for~} \forall v\in V(H)\setminus \left\{y_{2i}^0,y_{2j}^0\right\}\]
	
	Clearly, $h_1$ is a local antimagic 3-coloring of $H_1$, thus $\chi_{la}(H_1)\le 3$ which completes the proof.
\end{proof}

\begin{remark}
	If the graph obtained in Theorem~\ref{th-SM-m} has $r\ge 2$ components, the split and merge process above can reduce the number of components by at most 1 in each split-merge step. We may apply this split-merge process repeatedly to certain degree 4 vertices with odd subscripts to obtain non-isomorphic graph(s) too.
\end{remark}

Note that $G(m^{[n]})$ in Theorem~\ref{th-Gmn} and $G_{m,n}$ in Theorem~\ref{th-SM-m} admit local antimagic 3-colorings with induced vertex labels $a = 2nm+2n+1, b = 2nm^2+2nm+m+n$ and $c = 2nm^2+2nm+m+n+2$. We shall partition all of the vertices into three parts according to the induced vertex labels $a,b,c$ respectively. Certain graphs with local chromatic number 3 can be obtained by merging vertices in the same part. For $1\le j\le m$, let $V_{1,j} = \left\{v_{2i-1,j}\in V(G(m^{[n]}))\mid 1\le i\le \delta\right\}$ and $V_{2,j} = \left\{v_{2i,j}\in V(G(m^{[n]})\mid 1\le i\le \delta\right\}$.

\begin{theorem}\label{th-merge-v}
	If $n = 2rs$ for $r\ge 1, s\ge 2$, partition $V_{1,j}$ into $r$ blocks of size $s$ denoted by $V_{1,j}^k$ while partition $V_{2,j}$ into $V_{2,j}^k$. Let graph $H_2$ be obtained from $G(m^{[n]})$ by merging all vertices in $V_{1,j}^k$ into $v_{1,j}^k$ and all vertices in $V_{2,j}^k$ into $v_{2,j}^k$ individually where $1\le j\le m$ and $1\le k\le r$. Then $\chi_{la}(H_2)=3$.
\end{theorem}

\begin{proof}
	Since all vertices in the same blocks have no common neighbors, $H_2$ is simple and loopless.
	
	After merging, edges that are originally adjacent in $G(m^{[n]})$ are still adjacent in $H$. By an argument similar to that in Theorem~\ref{th-SM-m}, we also conclude that $\chi_{la}(H_2) \ge \chi(H) = 3$. By Lemma~\ref{lem-Gmn-tri}, $V(G(m^{[n]}))$ can be partitioned into $U_1$, $U_2$ and $V$. Since vertices we merge are in $V$, vertices in $U_1$ and $U_2$ are kept.
	
	Define $f: E(G(m^{[n]}))\to \left\{1,2,\cdots,3n\right\}$ as shown in Theorem~\ref{th-Gmn} which implies a bijective edge labeling $h_2$ for $H_2$ with 
	\[h_2^+(x) = 2nm^2+2nm+m+n\text{~for~}x\in U_1,\]
	\[h_2^+(y) = 2nm^2+2nm+m+n+2\text{~for~} y\in U_2,\]
	\[h_2^+(v_{i,j}^k) = s(2nm+2n+1)\text{~for~} i=1, 2, 1\le j\le m\text{~and~} 1\le k\le r.\]
	Clearly, $h_2$ is a local antimagic 3-coloring of $H_2$. Therefore $\chi_{la}(H_2)=3$. 
\end{proof}

\begin{remark}
	We also can attain graphs with chromatic number 3 through similar method for $G_{m,n}$.
\end{remark}

\begin{corollary}\label{cor-C(2n,n)-del}
	The graph $H$ such that every component is the join of an even cycle with $O_2$ has local antimagic chromatic number $3$.
\end{corollary}

\begin{proof}
	Actually the graph $H$ is the case $r=1$ and $m=1$ in Theorem~\ref{th-merge-v}, which is obtained from $G(1^{[n]})$ by merging all vertices in $V_{i,1}$ into $v_{i,1}^0$ individually for $i=1, 2$, then $\chi_{la}(H)=3$.
\end{proof}

\section{Local antimagic chromatic number of $K_{1,s}\times C_n$ with even $n\ge 4$}

Now we consider the direct product $K_{1,s}\times C_n$. For $K_{1,s}$, let vertex set $V(K_{1,s})=\{v,v_1,v_2,\cdots$ $,v_s\}$, edge set $E(K_{1,s})=\left\{vv_1,vv_2,\cdots, vv_s \right\}$. For $C_n$, let $C_n=u_1u_2\cdots u_nu_1$ with even $n\ge 4$. Therefore vertex set of $K_{1,s}\times C_n$ is set $V(K_{1,s})\times V(C_n)=\bigcup_{i=1}^{n}\bigcup_{k=1}^{s}\left\{{(v,u_i),(v_k,u_i)}\right\}$. Additionally, two vertices $(x,y)$ and $(u,v)$ are adjacent in $K_{1,s}\times C_n$ if and only if $xu\in E(K_{1,s})$ and $yv\in E(C_n)$. Let $\delta = n/2$ as above.

\begin{figure*}[htbp]
	\centering
	\begin{minipage}[b]{0.2\textwidth} 
		\centering
		\includegraphics[width=\textwidth]{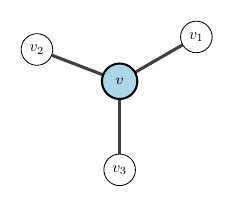}
		\caption{$K_{1,3}$}
		\label{fig: graph3}
	\end{minipage}
	\begin{minipage}[b]{0.15\textwidth} 
		\centering
		\includegraphics[width=\textwidth]{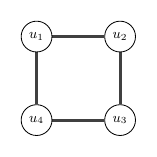}
		\caption{$C_4$}
		\label{fig: graph4}
	\end{minipage}		
\end{figure*}

Denote vertices $(v,u_{2j-1}), (v,u_{2j}), (v_k,u_{2j}), (v_k,u_{2j-1})$ in $K_{1,s}\times C_n$ by $u_j^1, u_j^2, v_{j,k}^1,$ $ v_{j,k}^2$ respectively, where $1\le j\le \delta$ and $k=1,2,\cdots,s$. Observe that $K_{1,s}\times C_n$ is two copies of $B(s^{[\delta]})$. For example, graph $K_{1,3}$ in Figure~\ref{fig: graph3}, $C_4$ in Figure~\ref{fig: graph4} and $K_{1,3}\times C_4$ in Figure~\ref{fig: graph5}.

\begin{figure*}[htbp]
	\centering
	\begin{minipage}[b]{0.5\textwidth} 
		\centering
		\includegraphics[width=\textwidth]{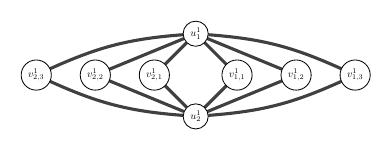}
	\end{minipage}\hfill 
	\begin{minipage}[b]{0.5\textwidth} 
		\centering
		\includegraphics[width=\textwidth]{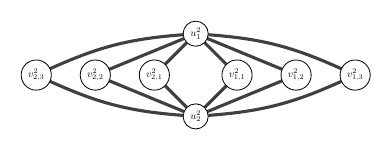}
	\end{minipage}\hfill 
    \caption{$K_{1,3}\times C_4$}
    \label{fig: graph5}

\end{figure*}

\begin{theorem}\label{th-K&C-even s}
	The graph $G$ with $p\ge 2$ components $B(s^{[n_i]})$ has $\chi_{la}(G)=2$ where $s$ is even, $1\le i\le p$, $\sum_{i=1}^{p}n_i=n$ and $n_i\ge 3$.
\end{theorem}

\begin{proof}
	As we mentioned in section 1, $B(s^{[n_i]})$ is a bipartite graph for each $1\le i\le p$. Suppose two partitions for each component $B(s^{[n_i]})$ are $\left(\bigcup_{j=1}^{n_i}X_j^i,\bigcup_{j=1}^{n_i}Y_j^i\right)$ where $X_j^i=\left\{u_j^i, u_{j+1}^i\right\}$, $Y_j^i=\left\{v_{j,1}^i, v_{j,2}^i,\cdots,v_{j,s}^i\right\}$ for $u_{n_i+1}^i=u_1^i$, $1\le i\le p$. The edge set of $B(s^{[n_i]})$ is $\bigcup_{j=1}^{n_i}E(K_{2,s}^j)$ where $K_{2,s}^j=(X^i_j,Y^i_j)$ is a complete bipartite graph. Now split $u_1^i$ into $u_{1,1}^i$ and $u_{1,2}^i$ for $1\le i\le p$ such that $u_{1,1}^i$ is adjacent to $v_{1,k}^i$ and $u_{1,2}^i$ is adjacent to $v_{n_i,k}^i$ with $1\le k\le s$. Next, merge $u_{1,2}^i$ and $u_{1,1}^{i+1}$ into $u_i^0$ for $1\le i\le p$ and $u_{1,1}^{p+1}=u_{1,1}^1$.
	
	The graph we obtain now is isomorphism to $B(s^{[n]})$ as expressed in section 1. If we define bijection similar to Lemma~\ref{lem-B}, $f: E(B(s^{[n]})) \to \left\{1,2,\cdots,2ns\right\}$ as 
	\[f(u_iv_{i,2j-1})=2i-1+n(2j-2),\]
	\[f(u_iv_{i,2j})=(2n-2i+1)+n(2j-2),\]
	\[f(u_{i+1}v_{i,2j-1})=2ns-2(i-1)-n(2j-2),\]
	\[f(u_{i+1}v_{i,2j})=2ns-(2n-2i)-n(2j-2),\]
	for $1\le i\le n$, $1\le j\le s/2$ with $f^+(u_i)=2ns^2+s, f^+(v_{i,k})=2ns+1$, $1\le k\le s$, then a bijection $g: E(B(G))\to \left\{1,2,\cdots,2sn\right\}$  with $g^+(v)=f^+(v)$ for $v\in V(G)\setminus \left\{u_i^0\mid 1\le i\le p\right\}$ is admitted. Now, $f$ has properties:
    \begin{equation*}
        \begin{aligned}
            \sum_{j=1}^{s/2}[f(u_iv_{i,2j-1})+f(u_iv_{i,2j})]=\delta s^2;
            \sum_{j=1}^{s/2}[f(u_iv_{i-1,2j-1})+f(u_iv_{i-1,2j})]=3\delta s^2+s.
        \end{aligned}
    \end{equation*}

	Therefore,  $g^+(u_i^0)=\delta s^2+3\delta s^2+s=2ns^2+s$ such that $g$ is a local antimagic 2-coloring which completes the proof.
\end{proof}

\begin{corollary}\label{cor-K&C-even s}
	Suppose $n\ge 4$ is even, we have $\chi_{la}(K_{1,s}\times C_n)=2$ if $s$ is even.
\end{corollary}

\begin{corollary}
	The graph $G$ with $p$ components $K_{1,s}\times C_{n_i}$ has $\chi_{la}(G)=2$ where $n_i$, $s$ are even and $\sum_{i=1}^{p}n_i=n$.
\end{corollary}

\begin{example}
	The graph $K_{1,4}\times C_6$ has local antimagic 2-coloring as in Figure~\ref{K_1,4timesC_6}.
	\begin{figure*}[htbp]
	    \centering
	\begin{minipage}[b]{0.3\textwidth} 
		\centering
		\includegraphics[width=\textwidth]{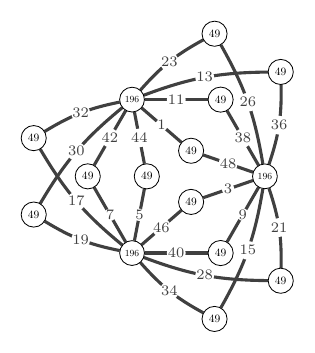}
	\end{minipage}
	\begin{minipage}[b]{0.3\textwidth} 
		\centering
		\includegraphics[width=\textwidth]{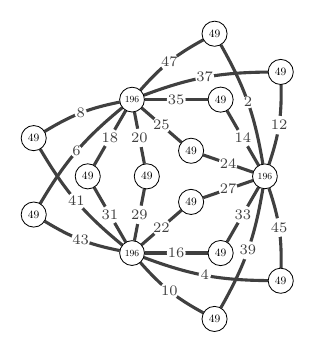}
	\end{minipage}	
    \caption{$K_{1,4}\times C_6$ with local antimagic 2-coloring}
    \label{K_1,4timesC_6}
	\end{figure*}

\end{example}


\begin{theorem}\label{th-K&C-n=2s}
	If $n\ge 4$ is even and $s=\delta$, then $\chi_{la}(K_{1,s}\times C_n)=2$.
\end{theorem}

\begin{proof}
	We will get Corollary~\ref{cor-K&C-even s} when $\delta$ is even so that we prove the theorem for $n \equiv 2 \pmod{4}$. We just need to prove $K_{1,s}\times C_n$ has a local $2$-antimagic labeling.
	
	Observe that the edge set of $K_{1,s}\times C_n$ decomposes into $2s$ subsets with size $n$, $E(K_{1,s}\times C_n)=\bigcup_{h=1}^2\bigcup_{k=1}^s\left\{u_1^h,v_{1,k}^h,u_2^h,v_{2,k}^h, \cdots, u_{\delta}^1,v_{\delta,k}^1\right\}$. And ordered sequences of all edges in each subset make up a cycle $C^h_{k}$ with consecutive edges $e_{1,k}^h, e_{2,k}^h, \cdots, e_{n,k}^h$ where $e_{1,k}^h=(u_k^h,v_{k,k}^h)$ (ordered sequence which means we can determine remaining edges) for $h=1,2$, $k\in [1,\delta]$. For instance, $K_{1,2}\times C_4$ is shown in Figure~\ref{fig: graph6}.
	
    \begin{figure*}[htbp]
	    \centering
    	\begin{minipage}[b]{0.4\textwidth} 
    		\centering
    		\includegraphics[width=\textwidth]{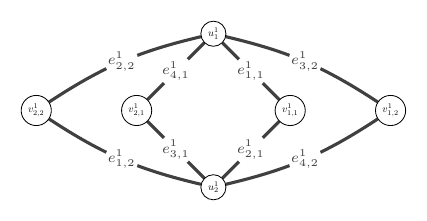}
    	\end{minipage}
    	\begin{minipage}[b]{0.4\textwidth} 
    		\centering
    		\includegraphics[width=\textwidth]{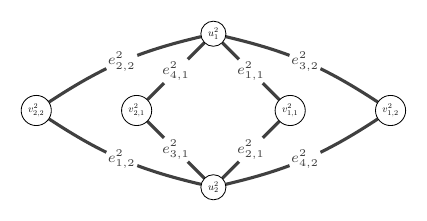}
    	\end{minipage}	
             \vspace{-4mm} 
    	\caption{$K_{1,2}\times C_4$}
    	\label{fig: graph6}
    \end{figure*}
	
	We label the edges according to $2s$ cycles $C^h_k$ where $h=1,2$ and $k\in [1,s]$. Now, define bijection $f: E(K_{1,s}\times C_n)\to n^2$ as
	\[f(e_{2j-1,k}^1)=n^2-2(k-1)n-2(j-1),\]
	\[f(e_{2j,k}^1)=2(k-1)n+2j-1,\]
	\[f(e_{2j-1,k}^2)=n^2-(2k-1)n-2(j-1),\]
	\[f(e_{2j,k}^2)=(2k-1)n+2j-1,\]
	where $j\in [1,\delta]$ and $k\in [1,s]$.
	
	As we can observe, 
	\[f(e_{1,i}^h)+f(e_{n,i}^h)=n^2+n+1\text{~for~} h=1,2\text{~and~} i\in [1,\delta],\]
	\[f(e_{2j,k}^h)+f(e_{2j+1,k}^h)=n^2-1\text{~for~} k\in[1,s],h=1,2\text{~and~} j\in[1,\delta-1],\]
	\[f(e_{2r-1,k}^h)+f(e_{2r,k}^h)=n^2+1\text{~for~}k\in[1,s], h=1,2\text{~and~}  r\in [1,\delta].\]
	Due to the properties of $K_{1,s}\times C_n$:
	\begin{enumerate}
		\item[(1)] In $C_i^h$, $u_i^h$ is incident to $e_{1,i}^h$ and $e_{n,i}^h$ because of the edge definition of $C_i^h$ for $i\in [1,\delta]$ and $h=1,2$;
		\item[(2)] For each $k\in [1,s]$, $k\ne i$, $h=1,2$, $\exists j\in [1,\delta]$ such that $u_i^h$ is incident to $e_{2j,k}^h$ and $e_{2j+1,k}^h$ in $C_k^h$;
		\item[(3)] For $h=1,2$, $k\in[1,s]$ and $i\in [1,\delta]$, $\exists r\in[1,\delta]$ such that $v_{i,k}^h$ is incident to $e_{2r-1,k}^h$ and $e_{2r,k}^h$,
		
	\end{enumerate}
	 $f^+(u_i^h)=n^2+n-1+(s-1)(n^2-1)$ and $f^+(v_{i,k}^h)=n^2+1$ where  $k\in [1,s]$, $h=1,2$ and $i\in [1,\delta]$. Therefore, $f$ is a local antimagic 2-coloring of $K_{1,s}\times C_n$.
\end{proof}

\newpage
\begin{example} We can label the edges of $K_{1,3}\times C_6$ in Figure~\ref{K_1,3timesC_6} as shown in Theorem \ref{th-K&C-n=2s}.

\begin{figure*}[ht]
	    \centering
    	\begin{minipage}[b]{0.28\textwidth} 
    		\centering
    		\includegraphics[width=\textwidth]{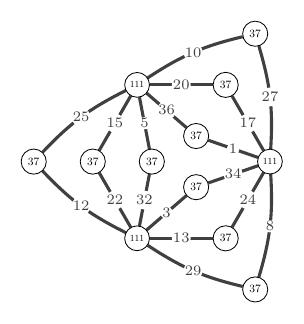}
    	\end{minipage}
    	\begin{minipage}[b]{0.28\textwidth} 
    		\centering
    		\includegraphics[width=\textwidth]{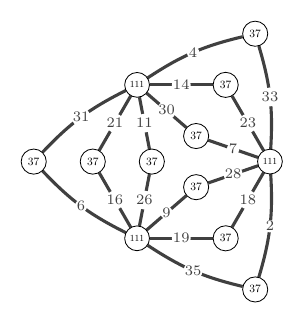}
    	\end{minipage}	
                   \vspace{-4mm} 
    	\caption{$K_{1,3}\times C_6$ with local antimagic 2-coloring}
    	\label{K_1,3timesC_6}
\end{figure*}

\end{example}

\begin{theorem}\label{th-K&C-merge}
	If $n\ge 4$ and $s$ are even and graph $G_1$ is obtained from $K_{1,s}\times C_n$ by merging $v_{i,k}^1$ and $v_{i,k}^2$ into $v_{i,k}$, graph $G_2$ is obtained from $K_{1,s}\times C_n$ by merging $u_i^1$ and $u_i^2$ into $u_i$ for $i\in [1,\delta]$, $k\in [1,s]$, then  $\chi_{la}(G_1)=\chi_{la}(G_2)=2$.
\end{theorem}

\begin{proof}
	
	Define $g: E(K_{1,s}\times C_n)\to \left\{1,2,\cdots,2ns\right\}$ as shown in Theorem \ref{th-K&C-even s} which indicates a bijection $h: E(G_1)\to \left\{1,2,\cdots,2ns\right\}$ with 
	\[h^+(v) = g^+(v)\text{ for }\forall v \in V(G_1)\setminus \left\{v_{i,k}\mid 1\le i\le \delta, 1\le k \le s\right\},\]
	\[h^+(v_{i,k}) = g^+(v_{i,k}^1)+g^+(v_{i,k}^2)= 2(2ns+1)=4ns+2\text{ for }1\le i\le \delta, 1\le k\le s.\]
	
	Therefore $h$ is a local antimagic
2-coloring which implies that $\chi_{la}(G_1)=2$.

Actually, the graph $G_2$ is bipartite with bipartition $(\bigcup_{i=1}^{\delta}X_i, \bigcup_{i=1}^{\delta}Y_i)$, where $X_i=\{u_i^1, u_{i+1}^1,$ $ u_i^2, u_{i+1}^2\}$, $Y_i = \left\{v_{i,1}, v_{i,2},\cdots, v_{i,s}\right\}$ and edge set $\bigcup_{i=1}^{\delta}E(K^i_{4,s})$ where $K^i_{4,s}$ is complete bipartite with partition $(X_i,Y_i)$ for $1\le i\le \delta$,  $u_{\delta+1}^1=u_1^1$ and $u_{\delta+1}^2=u_1^2$. This means $G_2$ is $B((2s)^{[\delta]})$ and $\chi_{la}(G_2)=2$ by Lemma \ref{lem-B}.   \end{proof}

\begin{example}
	Graph $G_1$ is obtained from $K_{1,4}\times C_6$ by merging $v_{i,p}^1$ and $v_{i,p}^2$ , graph $G_2$ is obtained by merging $u_i^1$ and $u_i^2$ for $i=1,2,3$ and $k=1,2,3,4$. Their local antimagic 2-colorings are shown as Figure~\ref{graph-K-C-merge-v} and Figure~\ref{graph-K-C-merge-u} respectively.

\begin{figure*}[ht]
	    \centering
    	\begin{minipage}[b]{0.39\textwidth} 
    		\centering
    		\includegraphics[width=\textwidth]{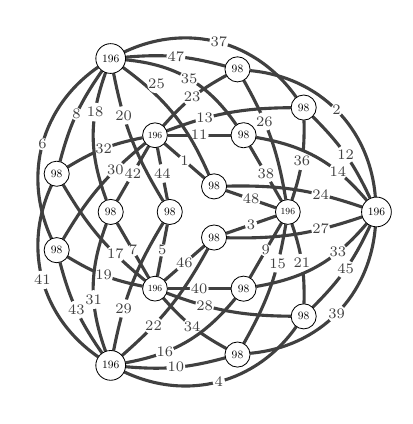}
         \caption{$G_1$}
            \label{graph-K-C-merge-v}
    	\end{minipage}
    	\begin{minipage}[b]{0.39\textwidth} 
    		\centering
    		\includegraphics[width=\textwidth]{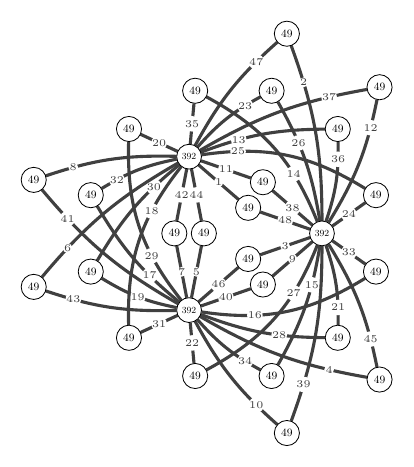}
                 \caption{$G_2$}
            \label{graph-K-C-merge-u}
    	\end{minipage}	
\end{figure*}
\end{example}

\newpage
\section{Conclusions and Open Problems}

 We showed that there are infinitely many bipartite and tripartite graphs with local antimagic chromatic number 2 and 3 respectively. Particularly, we initiated the study of the local antimagic chromatic number of the middle graph and the direct product of graphs. The following problems arise naturally.

\begin{prb}
	Determine $\chi_{la}(G(m^{[n]}))$ for even $m\ge 2$ and even $n\ge 4$.
\end{prb}

\begin{prb}
	Determine $\chi_{la}(G(m^{[n]}))$ for odd $n\ge 3$.
\end{prb}

\begin{prb}
	Determine $\chi_{la}(K_{1,s}\times C_n)$ for odd $n\ge 3$.
\end{prb}

\begin{prb}
        Determine $\chi_{la}(G)$ where $G$ is the middle graph of a graph.
\end{prb}

\begin{prb}
        Determine $\chi_{la}(G \times H)$ for various graphs $G$ and $H$.
\end{prb}


\providecommand{\bysame}{\leavevmode\hbox
to3em{\hrulefill}\thinspace}


\end{document}